\documentclass[11pt]{article}

\usepackage[margin=1in]{geometry}
\usepackage{amsmath,amssymb,amsthm,mathtools,mathrsfs}
\usepackage{booktabs}
\usepackage{graphicx}
\usepackage{microtype}
\usepackage[hidelinks]{hyperref}
\graphicspath{{./}}

\newtheorem{theorem}{Theorem}[section]
\newtheorem{proposition}[theorem]{Proposition}
\newtheorem{lemma}[theorem]{Lemma}

\newtheorem{remark}[theorem]{Remark}

\newcommand{\R}{\mathbb R}
\newcommand{\K}{\mathcal K}
\newcommand{\T}{\mathcal T}
\newcommand{\E}{\mathcal E}
\newcommand{\G}{G}

\title{Generalized Images for Brownian First Passage\\
to Monomial Boundaries}

\author{Gerardo Hern\'andez-del-Valle\thanks{Centro de Estudios Monetarios
Latinoamericanos (CEMLA), Mexico City, Mexico. Corresponding author.
E-mail: \texttt{ghernandez@cemla.org}}}

\date{}

\begin{document}
\maketitle

\begin{abstract}
We study the first passage of Brownian motion to a monomial moving boundary.
For every integer degree at least two, we construct a time-independent signed
image measure whose heat potential is the killed Brownian transition density
on a nonzero interval of the natural Brownian time scale.  The image density
is determined by a second-kind Volterra equation in the image coordinate,
not by the classical Volterra equations in time for the hitting density.  Its
kernel is an explicit Wright function, and the first-passage density is
recovered afterwards from the boundary flux.  Exponential bounds establish
convergence, while a beta-moment transform connects the Wright kernel with
the finite-order differential hierarchy previously obtained from
transported moving zeros.  A cubic numerical experiment independently
checks the resulting density against simulated Brownian paths.
\end{abstract}



\section{Introduction}

Let $(B_t)_{t\ge0}$ be standard Brownian motion with $B_0=x_0$, and let
$F$ be a continuous boundary with $F(0)>x_0$.  The first-passage time is
\[
 \tau_F=\inf\{t>0:B_t\ge F(t)\}.
\]
Closed expressions for its law are exceptional.  Classical approaches include
the method of images, integral equations, changes of measure, and moving-domain
heat equations; see, among others, Lerche~\cite{Lerche}, Durbin~\cite{Durbin},
Kahale~\cite{Kahale}, and Jaimungal, Kreinin and Valov~\cite{JKV}.

The inverse method of images asks whether a prescribed boundary admits a
representing measure.  A recent systematic treatment by Christensen, Hallmann
and Klein~\cite{CHK} develops a duality and linear-programming framework,
including sufficient conditions for positive representing measures and a
convergent numerical procedure.  Positivity is a substantive restriction.  In
the convex monomial problem considered here, the continuous part of the image
measure begins with a negative term.  We therefore work with signed exterior
images.  Our objective is not a general representability criterion, but an
explicit analytic construction for the family
\begin{equation}\label{eq:boundary}
 F(t)=d+\rho t^p,\qquad p\ge2.
\end{equation}

The present problem is the probabilistic continuation of the transported-zero
program for the heat semigroup.  If
\[
 P_t=\exp\!\left(\frac t2\partial_x^2\right),
\]
the inverse transported-zero problem prescribes $F$ and seeks initial objects
$\gamma$ satisfying
\begin{equation}\label{eq:tz}
 (P_t\gamma)(F(t))=0.
\end{equation}
It is useful to extend the inverse kernel explicitly to distributions and set
\begin{equation}\label{eq:distributional-kernel}
 \mathscr K_{\mathrm{dist}}(F;T)
 =\left\{\gamma\in\mathcal D'(\R):
 (P_t\gamma)(F(t))=0,\ 0<t<T\right\},
\end{equation}
whenever the displayed traces are defined.  Previous work studied
function-valued selections, stationary differential operators and the grading
$2p-1$ associated with monomial boundaries
\cite{HDVGuerra,HDVMatrices,HDVHeatAiry,HDVGrading,HDVInverse}.  Here we make
the distributional selection
\[
 \gamma=\delta_{x_0}-\mu,
\]
where $\mu$ is supported outside the initial domain.  Thus the new construction
belongs to the same inverse kernel, but its purpose is probabilistic: the heat
transport of $\gamma$ will be the killed transition density.

\subsection*{From the Airy spectral construction to the Volterra equation}

The route to the present equation is central to the result.  In the quadratic
case, the probabilistic initial condition can be imposed by a spectral method.
The stationary Airy family may be written, after normalization, in the form
\begin{equation}\label{eq:airy-spectral-family}
 \phi''(x)-(\kappa x+b)\phi(x)=0.
\end{equation}
Under heat transport, the parameter $b$ remains in the transported solution
but disappears from the compatibility equation that determines the quadratic
zero curve.  It therefore supplies a spectral degree of freedom without
altering the boundary.  The zeros of the Airy function give the eigenvalues of
the associated Sturm--Liouville problem, its eigenfunctions form a complete
orthogonal family, and the closure relation can be used to reproduce the
required Dirac-type initial datum.  This is the mechanism used in
\cite{HDVGuerra} to pass from an individual zero-producing heat solution to a
first-passage density.

The same route was not available for higher monomials.  The direct
transported-operator theory produced individual heat solutions with moving
zeros and the compatibility-matrix construction made the boundary equations
systematic \cite{HDVMatrices}.  The Heat--Airy extension reorganized those
compatibility conditions through an auxiliary commuting flow
\cite{HDVHeatAiry}.  Stationary grading then showed that the simplest
finite-order monomial ansatz does not continue unchanged beyond the low-degree
cases and identified $2p-1$ as the characteristic coefficient step
\cite{HDVGrading}.  Subsequent inverse analysis showed that a prescribed
monomial boundary determines a large kernel rather than a unique initial
function: sparse and constant-flux normalizations select different elements,
with different differential structures \cite{HDVInverse,HDVMonomial}.

Those developments solved the problem of constructing and classifying
zero-producing functions, but they did not supply the analogue of Airy
completeness needed for Brownian first passage.  For higher monomials we had no
known spectral parameter invisible to the boundary equation, no self-adjoint
Sturm--Liouville realization, and no closure relation that reconstructs
$\delta_{x_0}$.  This is the precise point at which the Volterra equation
enters.

Rather than first constructing a spectral family $\{\gamma_\lambda\}$ and
then searching for coefficients satisfying
\[
 \delta_{x_0}=\int c(\lambda)\gamma_\lambda\,d\lambda,
\]
we impose the probabilistic interior atom from the outset and solve directly
for the compensating exterior distribution:
\begin{equation}\label{eq:affine-selection}
 \gamma_{F,x_0}=\delta_{x_0}-\mu,
 \qquad \operatorname{supp}\mu\subseteq[d,\infty).
\end{equation}
Consequently, the central object of this paper is the affine slice
\begin{equation}\label{eq:affine-slice}
 \mathscr K_{\mathrm{dist}}(F;T)
 \cap
 \left\{\delta_{x_0}-\mu:
 \operatorname{supp}\mu\subseteq[d,\infty)\right\}.
\end{equation}
The boundary condition for \eqref{eq:affine-selection} is linear in $\mu$.
For a monomial boundary, scaling and Laplace inversion turn it into the
second-kind Volterra equation derived below.  The Volterra density therefore
plays, for general monomial boundaries, the role played by the
Sturm--Liouville expansion coefficients in the quadratic case.  It is not an
unrelated alternative method; it is the distributional realization of the
inverse transported-zero selection when the spectral construction is no
longer available.

This distinction is essential.  It was already known from the classical
Fortet formulation, and from our earlier work, that the first-passage density
itself satisfies a Volterra equation in time.  The contribution here is not a
new iterative scheme for that density.  Instead, the unknown in
\eqref{eq:volterra-density} is a time-independent signed density in the
exterior image coordinate.  Once this image density has been determined, the
first-passage density follows directly from a single boundary-flux formula.
To the best of our knowledge, the Wright-function kernel obtained below has
not previously been identified in the inverse method of images for Brownian
first passage to monomial moving boundaries.  Our novelty claim concerns this
image-coordinate equation and its connection with the transported-zero ODE
hierarchy, not the classical Volterra characterization of the hitting-time
density.

The same genealogy explains the reappearance of the integer $2p-1$.  In the
stationary theory it is the grading step.  In the present distributional
selection it determines
\[
 \beta=\rho a^{2p-1},\qquad
 \kappa\asymp|\beta|^{2/(2p-1)},\qquad
 t\asymp|\rho|^{-2/(2p-1)}.
\]
Thus the Volterra scaling is the image-theoretic counterpart of stationary
grading.

The main contributions are as follows.
\begin{enumerate}
\item The image equation for \eqref{eq:boundary} is converted, after a
Laplace change of variables, into an explicit Volterra equation of the second
kind.
\item The Volterra kernel is identified with a Wright-type solution of an
order-$p$ ODE, and a beta-moment transmutation connects it exactly to the
order-$2p-1$ hypergeometric stationary equation from the transported-zero
theory.
\item Factoring the Volterra operators at $\beta$ and $-\beta$ reveals a
defect whose coefficients satisfy a finite-difference cancellation theorem.
\item A compact--tail argument yields exponential-type control on the signed
image density.  The exponent $2/(2p-1)$ agrees with the Brownian balance
$|\rho|t^p\asymp\sqrt t$.
\item The resulting heat potential is identified with the killed Brownian
kernel, giving an actual first-passage density on a nonzero interval.
\end{enumerate}

The preceding papers enter this construction in distinct ways.  The
published quadratic analysis supplies the distributional normalization and
the spectral benchmark; compatibility matrices and Heat--Airy transport
organize the direct boundary equations; stationary grading identifies the
step $2p-1$ and the obstruction to the naive finite-order ansatz; the inverse
problem exposes the nonuniqueness of prescribed-zero data; and the monomial
hierarchy provides the constant-flux hypergeometric selection.  None of those
results constructs the exterior image measure or the killed Brownian kernel
for higher monomial boundaries.  Those are the new probabilistic objects of
the present paper.

We do not claim that the interval obtained below is maximal, nor that the
constant in the exponential-type estimate is sharp.  The quadratic defect
predicts a distinguished leading constant, but the proof uses a nonoptimal
majorant.

\section{The exterior image equation}

Write
\[
 x_0=d-a,\qquad a>0,
\]
and let
\[
 \G(t,x)=\frac1{\sqrt{2\pi t}}\exp\!\left(-\frac{x^2}{2t}\right).
\]
We seek a time-independent signed measure $\mu$ supported on $[d,\infty)$
such that
\begin{equation}\label{eq:heat-image}
 P(t,x)=\G(t,x-x_0)-\int_d^\infty\G(t,x-y)\,\mu(dy)
\end{equation}
satisfies $P(t,F(t))=0$.

Put $y=d+z$.  Cancelling the common factor generated by $\rho^2t^{2p}$ in
the boundary identity gives
\begin{equation}\label{eq:image-z}
 e^{-a^2/(2t)-a\rho t^{p-1}}
 =\int_0^\infty e^{-z^2/(2t)+\rho zt^{p-1}}\,\widetilde\mu(dz).
\end{equation}
For $\rho=0$, the solution is the classical reflected atom
$\widetilde\mu=\delta_a$.

Scale
\begin{equation}\label{eq:scaling}
 z=ar,\qquad t=a^2s,\qquad
 \lambda=\frac1{2s},\qquad \beta=\rho a^{2p-1},
\end{equation}
and push the measure forward under $r\mapsto u=r^2$.  Equation
\eqref{eq:image-z} becomes
\begin{equation}\label{eq:laplace-image}
 e^{-\lambda-\beta/(2^{p-1}\lambda^{p-1})}
 =\int_1^\infty
 e^{-\lambda u+\beta\sqrt u/(2^{p-1}\lambda^{p-1})}
 \,\nu_{\beta,p}(du).
\end{equation}

\begin{remark}
The image coordinate is
\begin{equation}\label{eq:correct-push}
 y=d+a\sqrt u.
\end{equation}
In particular, the atom $u=1$ is located at $d+a=2d-x_0$, the reflected
starting point.
\end{remark}

\section{Laplace inversion and the Volterra equation}

Put $q=p-1$.  For $u>v\ge1$, define
\begin{equation}\label{eq:kernel}
 K_{\beta,p}(u,v)
 =\sum_{j\ge1}
 \frac{\beta^jv^{j/2}(u-v)^{jq-1}}
 {2^{jq}j!\Gamma(jq)}.
\end{equation}
The series is locally uniformly convergent on the Volterra triangle.  Let
\[
 (\K_{\beta,p}h)(u)=\int_1^uK_{\beta,p}(u,v)h(v)\,dv,
 \qquad \T_{\beta,p}=I+\K_{\beta,p}.
\]

\begin{proposition}[Volterra form]\label{prop:volterra}
Equation \eqref{eq:laplace-image} is equivalent, in the class of locally
finite measures of suitable exponential type, to
\begin{equation}\label{eq:operator-image}
 \T_{\beta,p}\nu_{\beta,p}=\T_{-\beta,p}\delta_1.
\end{equation}
Writing
\[
 \nu_{\beta,p}=\delta_1+n_{\beta,p}(u)\,du,
\]
the density solves
\begin{equation}\label{eq:volterra-density}
 n_{\beta,p}(u)+\int_1^uK_{\beta,p}(u,v)n_{\beta,p}(v)\,dv
 =K_{-\beta,p}(u,1)-K_{\beta,p}(u,1).
\end{equation}
\end{proposition}

\begin{proof}
Expand the second exponential in \eqref{eq:laplace-image}.  The identity
\[
 \mathcal L^{-1}\{e^{-\lambda v}\lambda^{-jq}\}(u)
 =\frac{(u-v)_+^{jq-1}}{\Gamma(jq)}
\]
gives \eqref{eq:operator-image} term by term.  Local uniform convergence
justifies the interchange.  Separating the atom gives
\eqref{eq:volterra-density}.  Conversely, taking Laplace transforms recovers
\eqref{eq:laplace-image}.
\end{proof}

\begin{remark}[Two distinct Volterra equations]\label{rem:two-volterra}
Equation \eqref{eq:volterra-density} is not the classical Fortet equation for
the first-passage density.  In the present notation, the latter is
\begin{equation}\label{eq:fortet-general}
 \G(t,F(t)-x_0)
 =\int_0^t q(s)\G(t-s,F(t)-F(s))\,ds,
\end{equation}
a first-kind equation in the time variable whose unknown is the hitting-time
density $q$.  By contrast, \eqref{eq:volterra-density} is a second-kind
equation in the image coordinate $u$ whose unknown is the time-independent
signed exterior image density $n_{\beta,p}$.  After $n_{\beta,p}$ is found,
$q$ is obtained directly from the boundary flux \eqref{eq:explicit-flux}; no
iteration in $q$ is involved.  The Neumann series used below is only an
analytic device for proving invertibility on the tail, not the definition of
the construction or a required numerical algorithm.

For comparison, a classical second-kind equation for the hitting density is
also available under the usual differentiability assumptions; see
Durbin~\cite{Durbin}, Peskir~\cite{Peskir}, and Zipkin~\cite{Zipkin}.  If
$h(t)=F(t)-x_0$, it takes the form
\begin{align}
 q(t)={}&\left(\frac{h(t)}t-h'(t)\right)\G(t,h(t))\notag\\
 &+\int_0^t\left[
 h'(t)-\frac{h(t)-h(s)}{t-s}
 \right]\G(t-s,h(t)-h(s))q(s)\,ds.
 \label{eq:durbin-general}
\end{align}
For $h(t)=a+\rho t^p$, this becomes
\begin{align}
 q_{\rho,p}(t)={}&
 \left[\frac at-(p-1)\rho t^{p-1}\right]
 \G(t,a+\rho t^p)\notag\\
 &+\rho\int_0^t\left[
 pt^{p-1}-\frac{t^p-s^p}{t-s}
 \right]\G(t-s,\rho(t^p-s^p))q_{\rho,p}(s)\,ds.
 \label{eq:durbin-monomial}
\end{align}
Equations \eqref{eq:fortet-general} and \eqref{eq:durbin-monomial} both have
$q$ as their unknown and integrate over past time.  Neither is the
image-coordinate equation \eqref{eq:volterra-density}.
\end{remark}

Standard local Volterra theory gives a unique continuous solution of
\eqref{eq:volterra-density} on every interval $[1,U]$.  It also gives the
small-parameter expansion
\begin{equation}\label{eq:first-order}
 n_{\beta,p}(u)
 =-\frac{\beta}{2^{p-2}\Gamma(p-1)}(u-1)^{p-2}+O_{p,U}(\beta^2).
\end{equation}
Thus for $\rho>0$ the continuous image begins negatively; the construction
is genuinely signed.

\section{The hidden differential structure of the image kernel}

The Volterra kernel is not an arbitrary analytic kernel.  It is generated by
a Wright-type function whose differential equation is connected, by an
explicit beta-moment transform, with the finite-order stationary equation
obtained previously from the constant-flux transported-zero selection
\cite{HDVMonomial}.  We record that connection here.

Put
\[
 m=p-1,\qquad N=2p-1,
\]
and define Wright's generalized Bessel function in the normalization
needed here \cite{Wright}
\begin{equation}\label{eq:wright-seed}
 W_m(z)=\sum_{j\geq1}\frac{z^j}{j!\,\Gamma(mj)}.
\end{equation}
Then \eqref{eq:kernel} becomes
\begin{equation}\label{eq:kernel-wright}
 K_{\beta,p}(u,v)
 =\frac1{u-v}
 W_{p-1}\!\left(
 \frac{\beta\sqrt v\,(u-v)^{p-1}}{2^{p-1}}
 \right).
\end{equation}
In particular, if $r=u-v$ and
$k_{\beta,p}(r,v)=K_{\beta,p}(v+r,v)$, then termwise differentiation gives
\begin{equation}\label{eq:kernel-ode}
 r\,\partial_r^p k_{\beta,p}
 +p\,\partial_r^{p-1}k_{\beta,p}
 -\frac{(p-1)\beta\sqrt v}{2^{p-1}}k_{\beta,p}=0,
 \qquad r>0.
\end{equation}
Thus the degree of the monomial boundary is exactly the differential order
of the kernel seed.

Let $\theta=z\,d/dz$ and set
\begin{equation}\label{eq:kernel-operator}
 \mathcal M_p
 =P_p(\theta)-z,
 \qquad
 P_p(\theta)
 =\theta\prod_{a=0}^{p-2}
 \bigl[(p-1)(\theta-1)+a\bigr].
\end{equation}
The coefficient recursion in \eqref{eq:wright-seed} gives
\begin{equation}\label{eq:wright-ode}
 \mathcal M_pW_{p-1}=0.
\end{equation}

We next introduce the beta-moment operator
\begin{equation}\label{eq:beta-transform}
 (\mathcal B_pF)(z)
 =\int_0^1\!\int_0^1
 \frac{F\!\left(z\,t^p(1-t)^{p-1}s^{N}\right)}{1-t}
 \,ds\,dt,
 \qquad F(0)=0.
\end{equation}
On monomials it acts diagonally:
\begin{equation}\label{eq:beta-moments}
 \mathcal B_p(z^j)=\tau_jz^j,
 \qquad
 \tau_j
 =\frac{\Gamma(pj+1)\Gamma((p-1)j)}
 {\Gamma(Nj+2)},
 \qquad j\geq1.
\end{equation}

\begin{proposition}[Wright-to-hypergeometric transmutation]
\label{prop:wright-transmutation}
Define
\begin{equation}\label{eq:H-transform}
 H_p(z)=1+\mathcal B_pW_{p-1}(z).
\end{equation}
Then
\begin{equation}\label{eq:H-series}
 H_p(z)
 =\sum_{j\geq0}
 \frac{\Gamma(pj+1)}{j!\,\Gamma(Nj+2)}z^j.
\end{equation}
Moreover, with
\begin{equation}\label{eq:transported-operator-z}
 A_p(\theta)=\prod_{\ell=0}^{N-1}(N\theta+1-\ell),
 \qquad
 B_p(\theta)=\prod_{r=1}^{p-1}(p\theta+r),
\end{equation}
one has
\begin{equation}\label{eq:H-ode}
 \bigl[A_p(\theta)-pzB_p(\theta)\bigr]H_p=0.
\end{equation}
Consequently, after setting $z=-2^p\rho x^N$,
\begin{equation}\label{eq:Gamma-transform}
 \Gamma_{\rho,p}(x)
 =xH_p(-2^p\rho x^N)
\end{equation}
is precisely the constant-flux transported-zero function, and it satisfies
\begin{equation}\label{eq:transformed-stationary-ode}
 \left\{
 N^{p-1}D^N
 +p2^p\rho\prod_{r=1}^{p-1}
 \bigl[p\vartheta+Nr-p\bigr]
 \right\}\Gamma_{\rho,p}=0,
 \qquad \vartheta=xD.
\end{equation}
\end{proposition}

\begin{proof}
For $j\geq1$, the two integrations in \eqref{eq:beta-transform} give
\[
 \int_0^1t^{pj}(1-t)^{(p-1)j-1}\,dt
 \int_0^1s^{Nj}\,ds
 =\frac{B(pj+1,(p-1)j)}{Nj+1}=\tau_j.
\]
Substitution of \eqref{eq:wright-seed} therefore proves
\eqref{eq:H-series}.  To verify \eqref{eq:H-ode}, write the coefficient of
$z^j$ in $H_p$ as
\[
 h_j=\frac{\Gamma(pj+1)}{j!\,\Gamma(Nj+2)}.
\]
For $j\geq1$, direct cancellation gives
\[
 A_p(j)h_j=pB_p(j-1)h_{j-1}.
\]
The constant coefficient also vanishes because $A_p(0)=0$.  This proves
\eqref{eq:H-ode}.  Finally, the change of variables
$z=-2^p\rho x^N$ converts \eqref{eq:H-ode} into
\eqref{eq:transformed-stationary-ode}.
\end{proof}

The proposition gives a precise connection between the two constructions.
The image kernel is generated by an order-$p$ Wright equation, while its
beta-moment transform is the generalized hypergeometric solution annihilated
by the order-$2p-1$ stationary operator.  The passage
\[
 p\longmapsto p+(p-1)=2p-1
\]
is therefore encoded directly in the beta moments.  This result does not yet
express the complete Volterra solution $n_{\beta,p}$ as a spectral expansion
of transported-zero functions.  It does show, however, that the image kernel
and the previously constructed ODE solution share an exact differential
backbone rather than merely a common scaling exponent.

\section{The factorization defect}

Apply $\T_{-\beta,p}$ to \eqref{eq:operator-image} and set
\begin{equation}\label{eq:defect-def}
 \T_{-\beta,p}\T_{\beta,p}=I+\E_{\beta,p},
 \qquad
 g_{\beta,p}=\T_{-\beta,p}^{2}\delta_1.
\end{equation}
The factorized equation is
\begin{equation}\label{eq:factorized}
 (I+\E_{\beta,p})\nu_{\beta,p}=g_{\beta,p}.
\end{equation}

\begin{proposition}[Exact defect coefficients]\label{prop:defect}
The defect begins at order $\beta^2$.  If
$\E_{\beta,p}=\sum_{n\ge2}\beta^n\E_{n,p}$, then
\begin{align}
E_{n,p}(u,v)
={}&\frac1{2^{nq}}
\sum_{j=1}^{n-1}
\frac{(-1)^jv^{(n-j)/2}}
{j!(n-j)!\Gamma(jq)\Gamma((n-j)q)}\notag\\
&\quad\times\int_v^u
\bigl(w^{j/2}-v^{j/2}\bigr)
(u-w)^{jq-1}(w-v)^{(n-j)q-1}\,dw .
\label{eq:defect-coeff}
\end{align}
In particular,
\begin{align}
E_{2,p}(u,v)
=-\frac{\sqrt v}{2^{2p-2}\Gamma(p-1)^2}
\int_v^u&(\sqrt w-\sqrt v)
(u-w)^{p-2}
(w-v)^{p-2}\,dw .
\label{eq:E2}
\end{align}
\end{proposition}

\begin{proof}
Expand $\K_{-\beta,p}+\K_{\beta,p}+
\K_{-\beta,p}\K_{\beta,p}$ by total degree $n$.  In the composition term,
replace $w^{j/2}$ temporarily by $v^{j/2}$.  The beta integral then produces
\[
 \frac{v^{n/2}(u-v)^{nq-1}}{2^{nq}\Gamma(nq)}
 \sum_{j=1}^{n-1}\frac{(-1)^j}{j!(n-j)!}.
\]
Since
\[
 \sum_{j=1}^{n-1}\frac{(-1)^j}{j!(n-j)!}
 =-\frac{1+(-1)^n}{n!},
\]
this term cancels the degree-$n$ contribution of
$\K_{-\beta,p}+\K_{\beta,p}$.  Restoring
$w^{j/2}-v^{j/2}$ gives \eqref{eq:defect-coeff}; $n=2$ gives
\eqref{eq:E2}.
\end{proof}

The quadratic coefficient is nonpositive and satisfies
\begin{equation}\label{eq:E2-bound}
 |E_{2,p}(u,v)|
 \le \frac{p-1}{2^{2p-1}\Gamma(2p-1)}(u-v)^{2p-2}.
\end{equation}
Indeed, use
$\sqrt w-\sqrt v\le(w-v)/(2\sqrt v)$ and evaluate the beta integral.

\section{Finite-difference cancellation}

The factor $v^{(n-j)/2}$ in \eqref{eq:defect-coeff} appears dangerous as
$v\to\infty$.  The next lemma removes this apparent growth.

\begin{lemma}[Cancellation lemma]\label{lem:cancellation}
For $m\ge1$, define
\[
 S_{n,m}^{(q)}=
 \sum_{j=1}^{n-1}
 \frac{(-1)^j}{j!(n-j)!}
 \binom{j/2}{m}\bigl((n-j)q\bigr)_m.
\]
Then
\begin{equation}\label{eq:cancellation}
 S_{n,m}^{(q)}=0\qquad\text{whenever }2m<n.
\end{equation}
\end{lemma}

\begin{proof}
As a function of $j$,
\[
 \binom{j/2}{m}\bigl((n-j)q\bigr)_m
\]
is a polynomial of degree at most $2m$.  Its values at $j=0$ and $j=n$
vanish.  We may therefore add the two endpoint terms.  Up to the factor
$1/n!$, the resulting alternating binomial sum is the $n$th finite difference
of a polynomial of degree smaller than $n$, and hence is zero.
\end{proof}

Put $m_n=\lceil n/2\rceil$.  Taylor's theorem, followed by
Lemma~\ref{lem:cancellation}, yields the global estimate
\begin{equation}\label{eq:global-defect}
 |E_{n,p}(u,v)|
 \le
 \frac{b_{n,p}}{\Gamma(nq+m_n)}
 v^{n/2-m_n}(u-v)^{nq+m_n-1},
\end{equation}
where
\begin{equation}\label{eq:bnp}
 b_{n,p}=\frac{2^{n(2-p)}e^{m_n}(np)^{m_n}}{n!}.
\end{equation}
To see this, the Taylor remainder for $x^{j/2}$ is bounded by
\[
 e^{m_n}v^{j/2-m_n}(w-v)^{m_n};
\]
the remaining beta integral contributes
$((n-j)q)_{m_n}/\Gamma(nq+m_n)$, and
\[
 ((n-j)q)_{m_n}\le(np)^{m_n},\qquad
 \sum_{j=1}^{n-1}\frac1{j!(n-j)!}\le\frac{2^n}{n!}.
\]
Stirling's formula implies
\begin{equation}\label{eq:b-root}
 \limsup_{n\to\infty}b_{n,p}^{1/n}=0.
\end{equation}

The parity consequence is central:
\begin{equation}\label{eq:parity}
 v^{n/2-m_n}=
 \begin{cases}
 1,&n\ \text{even},\\
 v^{-1/2},&n\ \text{odd}.
 \end{cases}
\end{equation}

\section{Weighted tail estimates}

For $\kappa>0$ and $V\ge1$, let
\[
 X_{\kappa,V}=
 \left\{h:\|h\|_{\kappa,V}
 =\sup_{u\ge V}e^{-\kappa(u-V)}|h(u)|<\infty\right\}.
\]
Define the tail defect by integrating only over $V\le v\le u$.
From \eqref{eq:global-defect},
\begin{equation}\label{eq:tail-coeff}
 \|\beta^n\E_{n,p}^{(V)}\|_{\kappa,V}
 \le b_{n,p}|\beta|^n\kappa^{-(nq+m_n)}
 \begin{cases}
 1,&n\ \text{even},\\
 V^{-1/2},&n\ \text{odd}.
 \end{cases}
\end{equation}
Put
\[
 z=\frac{|\beta|}{\kappa^{p-1/2}},
\]
and define the entire functions
\[
 \Phi_p(z)=\sum_{k\ge1}b_{2k,p}z^{2k},
 \qquad
 \Psi_p(z)=\sum_{k\ge1}b_{2k+1,p}z^{2k+1}.
\]
Then
\begin{equation}\label{eq:tail-norm}
 \|\E_{\beta,p}^{(V)}\|_{\kappa,V}
 \le \Phi_p(z)+(\kappa V)^{-1/2}\Psi_p(z).
\end{equation}

Choose $z_p>0$ so that $\Phi_p(z_p)<1/4$, and put
\begin{equation}\label{eq:kappa-choice}
 \kappa_p(\beta)=
 \left(\frac{|\beta|}{z_p}\right)^{2/(2p-1)}.
\end{equation}
Choose $R_p$ so that $R_p^{-1/2}\Psi_p(z_p)<1/4$ and set
\begin{equation}\label{eq:V-choice}
 V_p(\beta)=\frac{R_p}{\kappa_p(\beta)}.
\end{equation}
It follows that the tail norm is smaller than $1/2$.

For completeness, the factorized forcing also lies in the weighted space.
The only nontrivial component is
\[
 H_{\beta,p}(u)=\int_1^u
 K_{-\beta,p}(u,w)K_{-\beta,p}(w,1)\,dw.
\]
The coefficient of $\beta^n$ satisfies, with $s=u-1$,
\begin{equation}\label{eq:forcing-bound}
 |H_{n,p}(u)|
 \le\frac{2^{n(3/2-q)}}{n!\Gamma(nq)}
 s^{n(q+1/2)-3/2},\qquad s\ge1.
\end{equation}
After multiplication by $e^{-\kappa s}$, Stirling's formula gives another
entire majorant in $z$.  The single-kernel part
$K_{-\beta,p}(u,1)$ is subexponential.  The compact-history term obtained by
integrating over $1\le v\le V_p(\beta)$ satisfies the same type of estimate.

\begin{lemma}[Compact--tail continuation]\label{lem:continuation}
Let $I+\mathcal A$ be a locally invertible Volterra operator.  Suppose its
tail restriction to $X_{\kappa,V}$ has norm smaller than one, and the forcing
after subtraction of the compact history belongs to $X_{\kappa,V}$.  Then the
local solution on $[1,V]$ extends uniquely to $[1,\infty)$, and its tail lies
in $X_{\kappa,V}$.
\end{lemma}

\begin{proof}
Solve first on $[1,V]$.  For $u\ge V$, move the integral over $[1,V]$ into
the forcing.  The remaining equation is inverted by the Neumann series on
$X_{\kappa,V}$.  Agreement at $V$ and uniqueness follow from the Volterra
form.
\end{proof}

\begin{theorem}[Image density]\label{thm:image-density}
For every integer $p\ge2$ and every $\beta\in\R$, equation
\eqref{eq:operator-image} has a unique locally finite solution
\[
 \nu_{\beta,p}=\delta_1+n_{\beta,p}(u)\,du.
\]
There are constants $C_p,C_{\beta,p}>0$ such that
\begin{equation}\label{eq:exp-type}
 |n_{\beta,p}(u)|
 \le C_{\beta,p}
 \exp\!\left(C_p|\beta|^{2/(2p-1)}u\right),
 \qquad u\ge1.
\end{equation}
For $\beta=0$, $n_{0,p}=0$.
\end{theorem}

\begin{proof}
For $\beta\ne0$, local Volterra theory gives the solution on
$[1,V_p(\beta)]$.  Estimates \eqref{eq:tail-norm}--\eqref{eq:forcing-bound}
and Lemma~\ref{lem:continuation} give a unique solution of the factorized
equation with \eqref{eq:exp-type}.  To see that no spurious solution was
introduced, set
\[
 R=\T_{\beta,p}\nu_{\beta,p}-\T_{-\beta,p}\delta_1.
\]
The factorized equation gives $\T_{-\beta,p}R=0$.  A homogeneous Volterra
equation has only the zero local solution, so $R=0$ on every compact interval.
Thus the original equation holds.
\end{proof}

\section{The killed heat kernel and first-passage density}

Let $\mu_{\rho,p}$ be the push-forward of $\nu_{\beta,p}$ under
$u\mapsto d+a\sqrt u$.  Explicitly, define
\begin{align}
P_{\rho,p}(t,x)
={}&\G(t,x-x_0)-\G(t,x-d-a)\notag\\
&-\int_1^\infty
\G(t,x-d-a\sqrt u)n_{\beta,p}(u)\,du.
\label{eq:final-heat}
\end{align}

\begin{theorem}[Monomial first passage]\label{thm:main}
Let $p\ge2$, $a>0$, $\rho\in\R$, $x_0=d-a$, and
$F(t)=d+\rho t^p$.  There exists $c_p>0$, depending only on $p$,
such that the following statements hold.  For $\rho\ne0$, define
\begin{equation}\label{eq:analytic-geometric-times}
 T_{\mathrm{an}}=c_p|\rho|^{-2/(2p-1)},
 \qquad
 T_{\mathrm{geom}}=
 \begin{cases}
  +\infty,&\rho\ge0,\\
  (a/|\rho|)^{1/p},&\rho<0.
 \end{cases}
\end{equation}
\eqref{eq:final-heat} converges, together with the derivatives required by the
heat equation, whenever
\begin{equation}\label{eq:time-window}
 0<t<T_{\mathrm{an}}.
\end{equation}
On the interval
\begin{equation}\label{eq:probabilistic-window}
 0<t<T_{\rho,a,p}:=T_{\mathrm{an}}\wedge T_{\mathrm{geom}},
\end{equation}
$P_{\rho,p}$ is the transition density of Brownian motion started at $x_0$
and killed at $F$.  The first-passage time has density
\begin{equation}\label{eq:flux}
 q_{\rho,p}(t)=-\frac12\partial_xP_{\rho,p}(t,F(t)^-),
\end{equation}
namely
\begin{align}
q_{\rho,p}(t)
={}&\frac{a+\rho t^p}{2t}\G(t,a+\rho t^p)
-\frac{\rho t^p-a}{2t}\G(t,\rho t^p-a)\notag\\
&-\frac1{2t}\int_1^\infty
(\rho t^p-a\sqrt u)\G(t,\rho t^p-a\sqrt u)
n_{\beta,p}(u)\,du.
\label{eq:explicit-flux}
\end{align}
For $\rho=0$, the construction is valid for all $t>0$ and reduces to the
classical constant-boundary L\'evy density
\[
 q_{0,p}(t)=\frac{a}{\sqrt{2\pi t^3}}
 \exp\!\left(-\frac{a^2}{2t}\right).
\]
\end{theorem}

\begin{proof}
By Theorem~\ref{thm:image-density}, the image integral is dominated by an
exponential with coefficient
\[
 -\frac{a^2}{2t}+C_p|\beta|^{2/(2p-1)}
\]
in $u$.  Since $|\beta|^{2/(2p-1)}=|\rho|^{2/(2p-1)}a^2$, it is negative
under \eqref{eq:time-window}, after decreasing $c_p$ if necessary.  Gaussian
domination permits differentiation under the integral.

At $x=F(t)$, cancellation of the common factor generated by
$\rho^2t^{2p}$ reduces $P_{\rho,p}(t,F(t))=0$ exactly to
\eqref{eq:laplace-image}.  The image support lies strictly to the right of the
initial domain, so $P_{\rho,p}(t,\cdot)\to\delta_{x_0}$ distributionally on
$(-\infty,d)$ as $t\downarrow0$.

On \eqref{eq:probabilistic-window} the boundary remains strictly above the
starting point.  Since $F$ is smooth, the stopped Brownian motion has a killed
transition density solving the moving Dirichlet problem, and its hitting time
has a continuous density; see, for example, Lee~\cite{Lee}.  This Dirichlet
problem is unique in the present Gaussian-decay class.  Indeed, subtract two
solutions, truncate the domain at $x=-R$, apply the parabolic maximum
principle, and let first $R\to\infty$ and then the initial time tend to zero.
The probabilistic killed kernel and $P_{\rho,p}$ have the same equation,
boundary trace and interior initial datum, so they coincide.  In particular,
the signed image formula represents a nonnegative sub-probability density.
Finally, differentiating the survival probability
\[
 \int_{-\infty}^{F(t)}P_{\rho,p}(t,x)\,dx
\]
and using the zero boundary value gives \eqref{eq:flux}.  Differentiating
\eqref{eq:final-heat} gives \eqref{eq:explicit-flux}.
\end{proof}

\begin{remark}[Natural scale]
The time exponent follows independently from Brownian scaling:
\[
 |\rho|t^p\asymp\sqrt t
 \quad\Longleftrightarrow\quad
 t\asymp|\rho|^{-2/(2p-1)}.
\]
The same exponent is forced analytically by the quadratic defect
$\beta^2\E_{2,p}$.
\end{remark}

\section{Numerical verification}

The proof is analytic; the following computations serve as independent checks
of the transforms and signs.  We discretized \eqref{eq:volterra-density} by a
midpoint product rule with $\beta=0.1$ and solved the resulting lower-triangular
linear system directly.  We then substituted the result both into
\eqref{eq:laplace-image} and into the classical Fortet equation
\begin{equation}\label{eq:fortet}
 \G(t,F(t)-x_0)
 =\int_0^t q_{\rho,p}(s)
 \G(t-s,F(t)-F(s))\,ds.
\end{equation}
The tested relative residual ranges are shown in Table~\ref{tab:numerics}.

\begin{table}[ht]
\centering
\caption{Independent residual checks for $\beta=0.1$.}
\label{tab:numerics}
\begin{tabular}{ccc}
\toprule
$p$ & Laplace identity & Fortet identity \\
\midrule
2 & $6.0\times10^{-6}$--$4.1\times10^{-5}$
  & $6.4\times10^{-6}$--$1.5\times10^{-5}$ \\
3 & $3.3\times10^{-7}$--$5.3\times10^{-7}$
  & $1.6\times10^{-7}$--$1.7\times10^{-7}$ \\
4 & $4.6\times10^{-12}$--$9.1\times10^{-7}$
  & $2.7\times10^{-12}$--$7.6\times10^{-9}$ \\
5 & $3.1\times10^{-13}$--$1.1\times10^{-6}$
  & $2.9\times10^{-13}$--$7.4\times10^{-9}$ \\
\bottomrule
\end{tabular}
\end{table}

The first perturbative coefficients also agree with
\eqref{eq:first-order}.  The relatively larger $p=2$ errors are consistent
with the fact that its leading kernel does not vanish on the diagonal; for
$p\ge3$ the kernel has the factor $(u-v)^{p-2}$.

\subsection{A cubic first-passage experiment}

We next compare the image flux directly with simulated Brownian first-passage
times.  Take
\begin{equation}\label{eq:numerical-cubic-parameters}
 d=0,\qquad x_0=-1,\qquad a=1,\qquad
 F(t)=0.1t^3.
\end{equation}
Thus $p=3$ and $\beta=\rho a^5=0.1$.  We solved
\eqref{eq:volterra-density} on $1\le u\le32$ with $1800$ midpoint cells.
The kernel series was truncated only after its next term was smaller than
$10^{-15}$ relative to the accumulated sum.  The resulting signed image
density was then substituted into \eqref{eq:explicit-flux}.

For the independent Monte Carlo calculation, we generated $120{,}000$
Brownian paths on $0\le t\le1.2$ with time step $10^{-3}$.  Endpoint
monitoring alone systematically misses crossings between adjacent grid
times.  We therefore used the conditional Brownian-bridge crossing
probability
\begin{equation}\label{eq:bridge-correction}
 \exp\!\left(-\frac{2g_0g_1}{\Delta t}\right),
\end{equation}
when the two endpoint gaps $g_0,g_1$ below the linearly interpolated boundary
were positive.  Over one time step, replacing the cubic boundary by its chord
introduces an error of smaller order than the time discretization used here.
The histogram is normalized by all simulated paths, rather than only by paths
that hit before the terminal time, so it estimates the unconditional density.

The observed hitting mass by time $1.2$ was
\[
 \widehat{\mathbb P}(\tau_F\le1.2)=0.33089167,
\]
whereas numerical integration of the generalized-image flux gave
\[
 \int_0^{1.2}q_{0.1,3}(t)\,dt=0.33128863.
\]
Their absolute difference is $3.97\times10^{-4}$, approximately $0.12\%$ of
the mass and well below the binomial Monte Carlo standard error
\[
 \sqrt{\frac{0.331(1-0.331)}{120000}}\simeq1.36\times10^{-3}.
\]
The root-mean-square discrepancy between the histogram and the theoretical
density over $0.15\le t\le1.15$ was $0.01319$.

Figure~\ref{fig:cubic-simulation} shows representative Brownian paths and the
density comparison.  The image density used to construct the red curve is
signed: on the truncated computational interval its values range from about
$-2.55$ to $0.62$.  Nevertheless, the heat solution and its first-passage
flux are nonnegative, as predicted by Theorem~\ref{thm:main}.

\begin{figure}[ht]
\centering
\includegraphics[width=\textwidth]{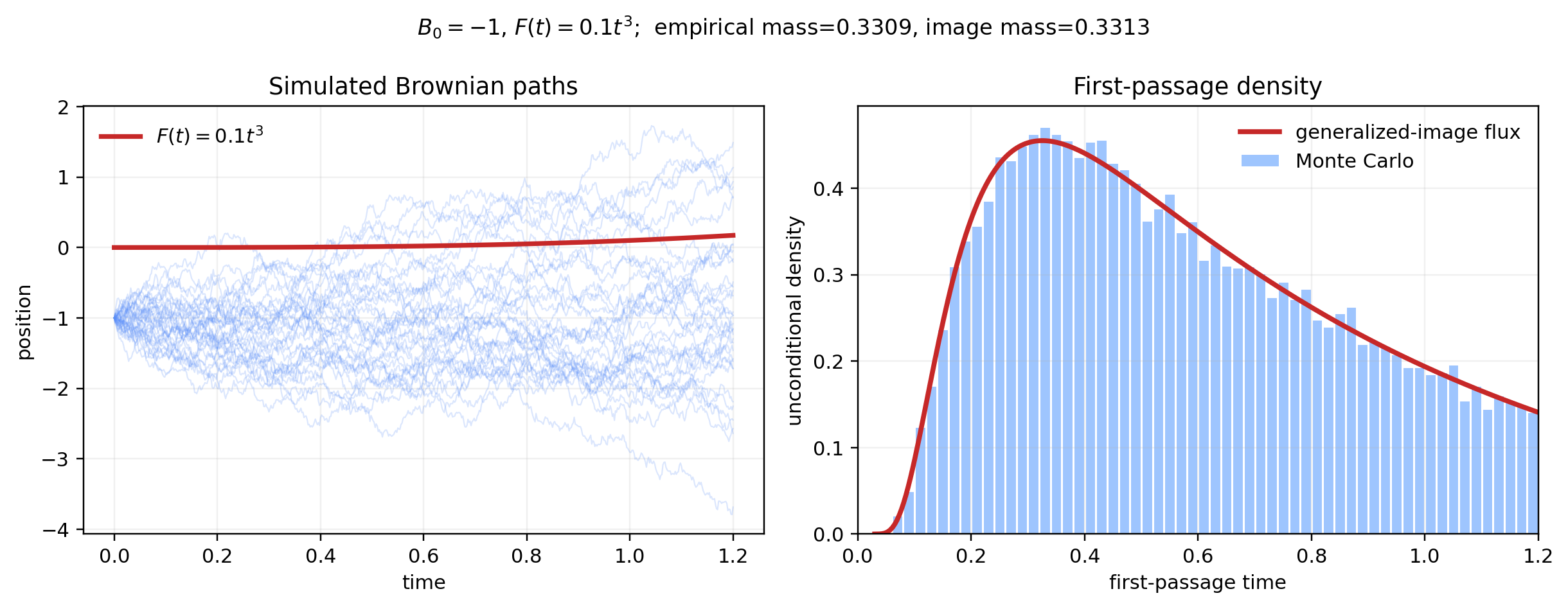}
\caption{Cubic first-passage experiment with $B_0=-1$ and
$F(t)=0.1t^3$.  Left: representative Brownian paths and the moving boundary.
Right: the unconditional first-passage histogram from $120{,}000$
bridge-corrected simulations and the generalized-image boundary flux.}
\label{fig:cubic-simulation}
\end{figure}

\section{Relation with inverse transported zeros}

Equation \eqref{eq:tz} defines an inverse kernel associated with the prescribed
curve $F$.  The development leading to the present construction began with
the observation that a heat solution vanishing on a moving curve can be
generated by transporting a stationary solution of an ordinary differential
equation.  In the quadratic problem, the spectral parameter of the Airy
equation is invisible to the boundary compatibility relation.  The resulting
Sturm--Liouville family reconstructs a Dirac initial datum and produces a
Brownian hitting-time density \cite{HDVGuerra}.  That argument is the
probabilistic model for the present paper, but no comparable completeness
relation was available for higher monomials.

The compatibility-matrix theory \cite{HDVMatrices} systematized the direct
problem: starting from a stationary differential operator, it combines its
heat transport with successive boundary identities to determine possible
moving zeros.  The Heat--Airy construction \cite{HDVHeatAiry} added a
commuting third-order flow and reorganized the same compatibility information
into a higher-flow hierarchy.  These developments explain how to derive
boundary equations, but they do not prescribe an initial distribution or
construct a killed transition kernel.

Stationary grading \cite{HDVGrading} then identified $2p-1$ as the natural
coefficient step for a degree-$p$ boundary and showed why the naive family
$D^p+\rho x^{p-1}$ ceases to produce an exact monomial zero in the first
nontrivial higher case.  This obstruction motivated the inverse formulation:
prescribe $F$ first and determine elements of its transported-zero kernel.
The inverse analysis \cite{HDVInverse} showed that this kernel is generally
large and that different normalizations can have different differential
structures.  Finally, the constant-flux normalization selected an explicit
hypergeometric function annihilated by an order-$2p-1$ polynomial-coefficient
operator for every $p\ge2$ \cite{HDVMonomial}.

All of those selections are function-valued.  They solve the zero-trace
problem, but they do not have the interior Dirac mass required for Brownian
first passage.  The present paper makes instead the distributional selection
\begin{equation}\label{eq:distributional-selection}
 \gamma_{x_0,F}=\delta_{x_0}-\mu_{\rho,p},
\end{equation}
where the second term is supported outside the initial domain.  Accordingly,
\[
 (P_t\gamma_{x_0,F})(F(t))=0.
\]
This normalization is the missing probabilistic ingredient: the unit interior
atom imposes the starting point, while the exterior distribution enforces
absorption without changing the initial data inside the domain.

The integer $2p-1$ appears in both theories.  In the stationary theory it is
the coefficient step of the monomial grading.  Here it controls
\[
 \beta=\rho a^{2p-1},\qquad
 \kappa\asymp|\beta|^{2/(2p-1)},\qquad
 t\asymp|\rho|^{-2/(2p-1)}.
\]
Thus the generalized-image construction is not a consequence of the
finite-order ODE theory, but it is a distributional branch of the same inverse
transported-zero program.  Proposition~\ref{prop:wright-transmutation} makes
the relationship exact: the Wright seed of the image-coordinate Volterra
kernel is carried by beta moments to the hypergeometric constant-flux
selection of \cite{HDVMonomial}.  The earlier ODE hierarchy therefore reveals
the differential structure hidden in the probabilistic image construction,
while the image measure supplies the distributional initial condition that
the function-valued theory lacked.

\section{Concluding remarks}

The construction gives a signed exterior image and a first-passage density for
every monomial degree on a nonzero, naturally scaled interval.  Several
questions remain.  The constant in the exponential-type bound is not sharp;
the exact quadratic defect suggests that the leading candidate is
\[
 \frac{(p-1)^{1/(2p-1)}}2|\beta|^{2/(2p-1)}.
\]
It is also open whether the representation can be continued for all times,
whether the image distribution changes analytic type at the predicted scale,
and which nonmonomial boundaries admit an analogous factorization.  Finally,
the signed nature of the image measure for positive convex monomials deserves
comparison with positive-measure representability in the classical inverse
method of images.

Finally, the construction should not be interpreted as another Volterra
scheme for solving directly for the first-passage density.  Classical
first-kind and second-kind equations in time already provide such
characterizations.  Here the time-independent exterior image measure is
constructed once through the second-kind equation
\eqref{eq:volterra-density}; the first-passage density is then recovered by
one evaluation of its boundary flux.  The familiar Fortet equation is used
only as an independent identity and numerical check.  What is new is the
Wright-kernel equation in the image coordinate and the differential
structure linking that kernel to the transported-zero hierarchy.

\section*{Acknowledgements}

The author used OpenAI's ChatGPT during the development of this work to
assist with symbolic exploration, numerical implementation, and language
editing.  All mathematical statements, computations, references, and
conclusions were reviewed by the author, who takes full responsibility for
the content of the manuscript.

\section*{Competing interests}

The author declares no competing interests.

\section*{Data availability}

No external data sets were used.  The numerical experiment uses simulated
Brownian paths and the algorithms described in Section~8.


\begin{thebibliography}{99}

\bibitem{CHK}
S.~Christensen, O.~Hallmann and M.~Klein,
\newblock On first passage time problems of Brownian motion---the inverse
method of images revisited,
\newblock \emph{arXiv:2404.16615}, 2024.

\bibitem{Durbin}
J.~Durbin,
\newblock The first-passage density of the Brownian motion process to a curved
boundary,
\newblock \emph{J. Appl. Probab.} \textbf{29} (1992), 291--304.

\bibitem{HDVGuerra}
G.~Hern\'andez-del-Valle and W.~A. Guerra-Polania,
\newblock On the heat equation with a moving boundary and applications to
hitting times for Brownian motion,
\newblock \emph{Stoch. Anal. Appl.} \textbf{42} (2024), 201--218.

\bibitem{HDVMatrices}
G.~Hern\'andez-del-Valle,
\newblock A recursive representation of compatibility matrices for
transported differential operators,
\newblock \emph{arXiv:2608.11365}, 2026.

\bibitem{HDVHeatAiry}
G.~Hern\'andez-del-Valle,
\newblock Heat--Airy transport of differential operators and moving
boundaries,
\newblock \emph{arXiv:2608.19434}, 2026.

\bibitem{HDVGrading}
G.~Hern\'andez-del-Valle,
\newblock Stationary grading and heat-transported moving boundaries: from the
cubic case to two fourth-order branches,
\newblock \emph{arXiv:2609.27914}, 2026.

\bibitem{HDVInverse}
G.~Hern\'andez-del-Valle,
\newblock Inverse transported-zero problems for the heat semigroup:
finite-order and non-$D$-finite quartic selections,
\newblock preprint, 2026.

\bibitem{HDVMonomial}
G.~Hern\'andez-del-Valle,
\newblock Monomial moving zeros for the heat semigroup: an explicit
finite-order differential hierarchy,
\newblock \emph{arXiv:2610.27326}, 2026.

\bibitem{JKV}
S.~Jaimungal, A.~Kreinin and A.~Valov,
\newblock Integral equations and the first passage time of Brownian motions,
\newblock \emph{Methodol. Comput. Appl. Probab.} \textbf{11} (2009),
757--774.

\bibitem{Kahale}
N.~Kahale,
\newblock Analytic crossing probabilities for certain barriers by Brownian
motion,
\newblock \emph{Ann. Appl. Probab.} \textbf{18} (2008), 1424--1440.

\bibitem{Lee}
J.~Lee,
\newblock First passage time densities through H\"older curves,
\newblock \emph{ALEA Lat. Am. J. Probab. Math. Stat.} \textbf{15} (2018),
1137--1155.

\bibitem{Lerche}
H.~R. Lerche,
\newblock \emph{Boundary Crossing of Brownian Motion: Its Relation to the Law
of the Iterated Logarithm and to Sequential Analysis},
\newblock Lecture Notes in Statistics 40, Springer, 1986.

\bibitem{Peskir}
G.~Peskir,
\newblock On integral equations arising in the first-passage problem for
Brownian motion,
\newblock \emph{J. Integral Equations Appl.} \textbf{14} (2002), 397--423.

\bibitem{Wright}
E.~M. Wright,
\newblock The asymptotic expansion of the generalized Bessel function,
\newblock \emph{Proc. London Math. Soc.} (2) \textbf{38} (1935), 257--270.

\bibitem{Zipkin}
P.~Zipkin,
\newblock On Durbin's series for the density of first passage times,
\newblock \emph{J. Appl. Probab.} \textbf{48} (2011), 713--722.

\end{thebibliography}
\end{document}